\documentclass{amsart}

\usepackage{graphicx}
\usepackage{float}
\usepackage{placeins}
\usepackage{mdframed}

\newtheorem{theorem}{Theorem}[section]
\newtheorem{lemma}[theorem]{Lemma}

\theoremstyle{definition}

\theoremstyle{remark}
\newtheorem{remark}[theorem]{Remark}

\numberwithin{equation}{section}

\begin{document}

\title{Random Polygon to Ellipse: A Generalization}

\author{Keller VandeBogert}
\email{kv00767@georgiasouthern.edu}
\begin{abstract}
    This paper generalizes the result of Elmachtoub et al to any weighted barycenter, where a transformation is considered which takes an arbitrary point of division $\xi \in (0,1)$ of the segments of a polygon with $n$ vertices. We then consider connecting these new points to form another polygon, and iterate this process. After considering properties of our generalized transformation matrix, a surprisingly elegant interplay of elementary complex analysis and linear algebra is used to find a closed form for our iterative process. We then specify the new limiting ellipse, $\mathcal{E}$, which has oscillating semi-axes. Along the way we find that the case for $\xi = 1/2$ enjoys some special optimality conditions, and periodicity of the ellipse $\mathcal{E}$ is analyzed as well. To conclude, an even more generalized case is considered: taking a different point of division for every segment of our polygon $\mathcal{P} (\vec{x}^{(0)}, \vec{y}^{(0)})$.
\end{abstract}

\date{\today}


\maketitle

\section{Introduction}

Consider a polygon $\mathcal{P} (\vec{x}^{(0)}, \vec{y}^{(0)})$, whose $x$ and $y$ coordinates are given by the vectors $\vec x$ and $\vec y$, respectively. Consider a transformation on $\mathcal{P} (\vec{x}^{(0)}, \vec{y}^{(0)})$ which creates a new polygon $\mathcal{P} (\vec{x}^{(1)}, \vec{y}^{(1)})$ by taking the midpoints of the original, connecting them, and then normalizing the resulting vertex vectors to unit length. It has been shown that iteration of this process produces a sequence of polygons $\mathcal{P} (\vec{x}^{(k)}, \vec{y}^{(k)})$ converging to an ellipse oriented at $45$ degrees \cite{orig}. 

Our goal is to relax the condition of taking only the midpoint, and allow the point of division $\xi$ to vary inside the open unit interval. Indeed, there is an elegant and surprising invariance in the orientation of our ellipse that hints at a deeper underlying structure for this type of problem. We note that this paper will make absolutely no assumption on our initial polygon. We do not require convexity or any other special condition that makes our polygon well-behaved. This is truly an example of order out of pure chaos. Indeed, there are numerous papers written on the subject of iterative transformations of random polygons, in which limiting behavior is of interest. Many of these papers hint at an underlying invariance and structure of these types of iterations (see \cite{orig}, \cite{3}, \cite{4}, and \cite{5}).

\section{The Matrix $M_n$}
In proving the case for $\xi = 1/2$ a matrix $M_n$ was introduced such that if $\vec x^{(k)}$ is the x coordinates of our vertex vectors after the kth iteration, then $M_n \vec x^{(k)} = \vec x^{(k+1)}$. We then wish to study the sequence $\mathcal{P}(M_n^k \vec x^{(0)}, M_n^k \vec y^{(0)})$, where $\mathcal{P}(\vec x, \vec y)$ denotes the polygon $\mathcal{P}$ with vertex coordinates given by $x$ and $y$, respectively. Since the polygon is entirely determined by these two vectors, we shall only study how each vertex vector changes with respect to our transformation. In this way, we find that our question of convergence becomes a question of the properties of iterating $M_n$. Following in these tracks, we introduce a more general version that applies to any $\xi$. 
Let 
\begin{equation}
\label{Mn}    
M_n =  
\begin{bmatrix}
1-\xi & \xi & 0 &.&.&.&0\\
0& 1-\xi&\xi &0 & . & .&0\\
0& 0 & 1-\xi &\xi &.&.&0\\
.& & & & & .&.\\
.& & & & & .&\xi\\
\xi& 0& 0& 0& 0& 0& 1-\xi 
\end{bmatrix}
 \end{equation}

With $M_n$ defined as so, we can turn to the analysis of its properties. Without normalizing after every iteration, it can be shown that iteration of $M_n$ will produce a sequence that eventually converges to the centroid of our original polygon.

\begin{lemma} The sequence of vertex vectors $\{ \vec{x}^{(k)} \}$ and $\{ \vec{y}^{(k)} \}$ created by iteration of $M_n$ has the same centroid for every $k \to \infty$.
\end{lemma}

\begin{proof}
Letting $e$ represent an $nx1$ matrix of 1's, we note that the centroid of the $x$ coordinates of our polygon is just $\frac{e^T \vec x}{n}$, and the centroid of the $y$ coordinates is defined analogously. We then see:
$$\frac{e^T \vec x^{(k)}}{n} = \frac{e^T M_n \vec x^{(k-1)}}{n}= \frac{e^T \vec x^{(k-1)}}{n}$$
\end{proof}

From this, we deduce that if our sequence converges to a point, that point is the centroid of our initial polygon, and that every polygon in the sequence has the same centroid. Intuitively, it is clear that 

$$ \Big[ \min \{ \vec{x}^{(k+1)} \}, \max \{ \vec{x}^{(k+1)} \} \Big] \subset \Big[ \min \{ \vec{x}^{(k)} \}, \max \{ \vec{x}^{(k)} \} \Big] $$

Where $\max\{ \vec x \}^{(k)}$ denotes the largest component of the vector $\vec{x}^{(k)}$ (and the minimum defined analogously). If we had a contraction after every iteration, it is obvious that the sequence of nested intervals must converge to a point. Not every case will induce a contraction, though, with the obvious example being a square. However, since our polygon must have a finite number of vertices, after a sufficiently large amount of iterations, our transformation intuitively should induce a contraction. Indeed, we will make this intuition more rigorous in Section 5.

By introducing the upshift matrix $S_n$ \cite{orig}, which is just the identity matrix with all of its entries shifted up once, we see that $$M_n = \Big( (1-\xi)I_n + \xi S_n \Big)$$

In this way, we see that $M_n$ shares the same eigenvectors as $S_n$ but with different eigenvalues. Since $S_n$ is unitary, its eigenvectors will be orthogonal. By normalizing we create an orthonormal set of eigenvectors where the eigenvalues of $S_n$ are $n$th roots of unity. 

Allow that $ \omega_j = e^{\frac{2 \pi i j}{n}} $ and

$$ v_j = \sqrt{\frac{1}{n}} 
\begin{bmatrix}
 &1  \\
 &\omega_j\\
 &\omega_j^2\\
 & .\\
 & .\\
 & \omega_j^{n-1}
\end{bmatrix}$$

It is easy to see that $S_n v_j = \lambda_j v_j$. In this way, we can also see that $M_n v_j = (1-\xi + \xi \omega_j)v_j$ so that the eigenvectors of $S_n$ are also eigenvectors of $M_n$. The eigenvalues of $M_n$ lie on the circle of radius $\xi$ centered at $1-\xi$ on the complex plane.

\begin{figure}
    \centering
\includegraphics[scale=0.48]{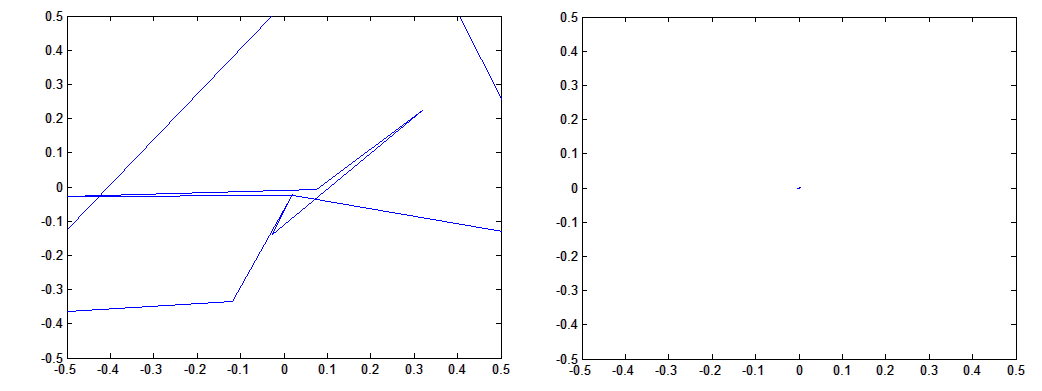}

\includegraphics[scale=0.48]{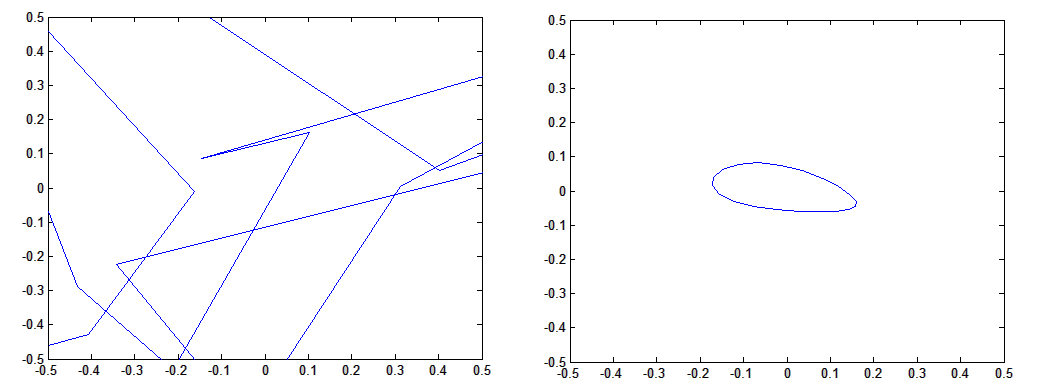}

\includegraphics[scale=0.48]{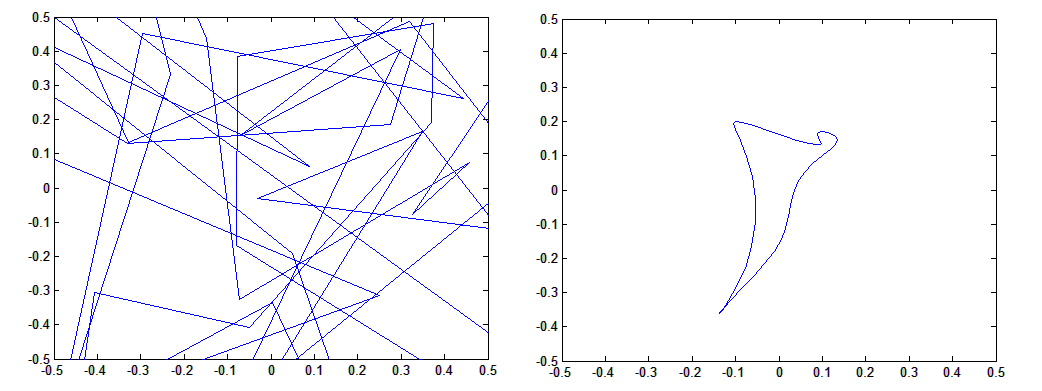}
\vspace{1mm}
\caption{The above illustrates the unnormed transformation after 100 iterations with $\xi = 1/4$. From top to bottom the initial polygons have $n=10$, $20$, and $50$ vertices, respectively. We see that the transformation does seems to tend to the centroid, and smaller $n$ seem to converge faster.}
\end{figure}

\section{The Damping Factor $\rho_n$}

From \cite{orig} we see that our analysis leads in the direction of repeated iteration of $M_n$ onto a unit vector with centroid zero in the span of our eigenvectors. We can assume that our unit vector has centroid zero because we can always just redefine the centroid as the origin of our coordinate system. This seemingly innocent assumption simplifies our analysis because it in fact forces our eigenvectors to be a dependent set, and the eigenvector corresponding to the eigenvalue $1$ turns out to be redundant.

It can then be seen that our vectors will becomes increasingly rich in the second largest eigenvalue and its complex conjugate. The rate at which our other eigenvectors converge can be measured by a \textit{damping factor} $\rho_n$ defined as the magnitude of the ratio of the third largest eigenvalue to the second largest. The reason we consider this is based off of an approximation method known as the \textit{power method}. Consider:

$$\rho_n =\bigg| \frac{\lambda_4}{\lambda_2} \bigg| = \max\bigg\{ \bigg| \frac{\lambda_4}{\lambda_2} \bigg|, \bigg| \frac{\lambda_5}{\lambda_2} \bigg|, ... ,\bigg| \frac{\lambda_n}{\lambda_2} \bigg| \bigg\}$$

This definition makes the most sense if we reorder our eigenvalues in such a way that $$|\lambda_2| = |\lambda_3| \geq |\lambda_4| = |\lambda_5| \geq ... = |\lambda_n|$$

We now seek to find an explicit damping factor.

$$\rho_n = \bigg| \frac{\lambda_4}{\lambda_2} \bigg| = \Bigg[ \frac{(1-\xi + \xi \cos(4\pi/n))^2 + \xi^2 \sin^2(4\pi/n)}{(1-\xi + \xi \cos(2\pi/n))^2 + \xi^2 \sin^2(2\pi/n)} \Bigg] ^{\frac{1}{2}}$$

With some basic algebra this can be simplified to:

\begin{equation}
\label{damp}
\rho_n = \Bigg[ \frac{1/\xi-4(1-\xi)\sin^2(2\pi/n)}{1/\xi-4(1-\xi)\sin^2(4\pi/n)} \Bigg]^{\frac{1}{2}}
\end{equation}

With \eqref{damp}, we deduce the following:

\begin{lemma}
For any $\xi \in (0,1)$, $\rho_n$ attains a minimum at $\xi = 1/2$. 
\end{lemma}

\begin{proof}
Differentiate $\rho_n$ with respect to $\xi$. 

$$\frac{d\rho_n}{d\xi} = 0  $$
We have:
$$ \big( ( - \frac{1}{\xi^2} + 4 \sin^2(2\pi/n)\big) \big( \frac{1}{\xi} - 4(1-\xi)\sin^2(4\pi/n) \big) = \big( ( - \frac{1}{\xi^2} + 4 \sin^2(4\pi/n)\big) \big( \frac{1}{\xi} - 4(1-\xi)\sin^2(2\pi/n) \big)$$

After some simplification, we are left with:
$$4 \Big( \frac{1}{\xi^2} - \frac{1}{\xi} \Big) \sin^2(4\pi/n) + \frac{4}{\xi}\sin^2(2\pi/n) = 4 \Big( \frac{1}{\xi^2} - \frac{1}{\xi} \Big) \sin^2(2\pi/n) + \frac{4}{\xi}\sin^2(4\pi/n)$$

$$\Rightarrow 1 - \xi = \xi \Rightarrow \xi = 1/2 $$
\end{proof}

\begin{remark}
It is interesting to note that the above proof does not specify any value for $n$. Indeed, this condition holds for any arbitrary $n$.
\end{remark}
With the damping factor now specified, it is clear that upon iteration the damping factor will eventually converge to 0. Heuristically, the damping factor makes sense because as our sequence tends to infinity, the largest eigenvalues will of course dictate the overall behavior of the sequence. By the above lemma we see that the midpoint yields the fastest possible rate of convergence. We now intend to specify what this sequence is actually converging to.

\begin{figure}[H]
    \centering
\includegraphics[scale=0.48]{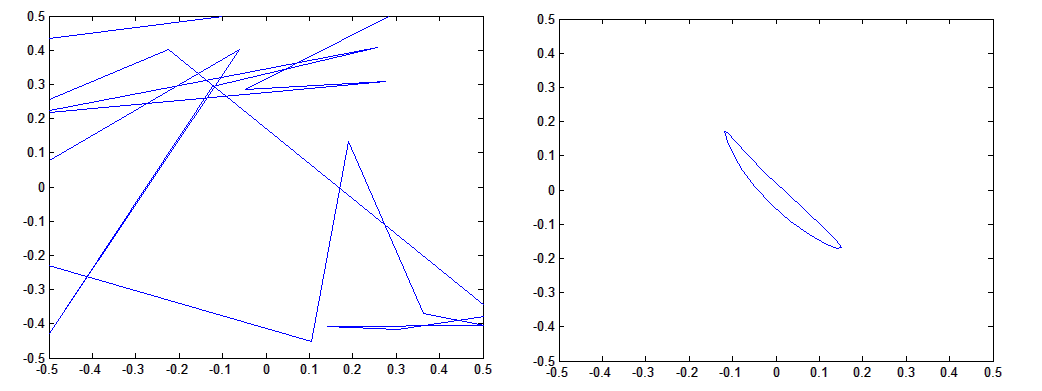}

\includegraphics[scale=0.48]{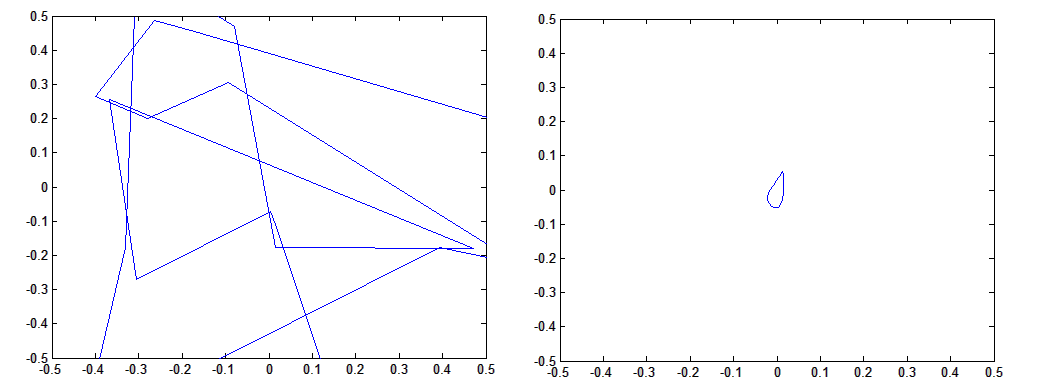}
\end{figure}

\begin{figure}[H]
\centering
\includegraphics[scale=0.48]{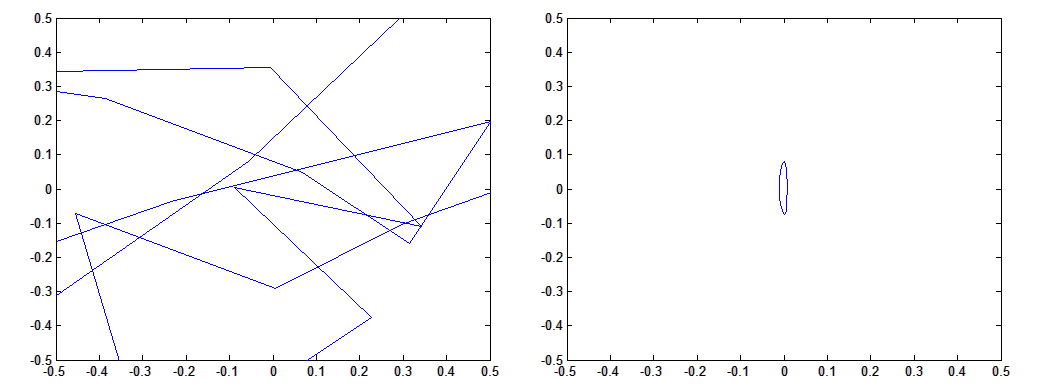}
\vspace{1mm}
\caption{The above illustrates the unnormed transformation after 100 iterations and 20 vertices with varying values for $\xi$. From top to bottom we have chosen $\xi = 1/5$, $1/4$, and $2/5$. It can be seen that as $\xi$ tends closer to $1/2$, the rate of convergence is faster.}
\end{figure}

\section{Convergence to $\mathcal{D}_2$}

We introduce an invariant subspace under our operator $M_n$ which is defined as $\mathcal{D}_2 = \textrm{span} \{ \textrm{Re}(z_2), \textrm{Im}(z_2) \} = \textrm{span} \{ \textrm{Re}(z_3), \textrm{Im}(z_3) \}$ where $z_j$ is a reordering of our eigenvectors in such a way that $M_n z_j = \lambda_j z_j$.

With the material of the previous section, we can introduce the undamped (unit normed) vector $$\tilde w^{(k)} =  \gamma_2 \bigg( \frac{\lambda_2}{|\lambda_2|} \bigg)^k z_2 + \gamma_3 \bigg( \frac{\lambda_3}{|\lambda_2|} \bigg)^k z_3$$

Where $\gamma_2 =  \bar{\gamma_3} \in \mathbb{C}$ are arbitrary constants. In this way, we define $$w^{(k)} = \frac{M_n w^{(k-1)}}{||M_n w^{(k-1)}||}$$ where $w^{(0)}$ is a centroid zero unit vector in $\textrm{span}(z_2,z_3, ... z_n)$. With this, we see that $$w^{(k)} = \tilde w^{(k)} + \mathcal{O}(\rho_n^k)$$

Since $\rho_n < 1$ for all $\xi \in (0,1)$, it is clear that our vector must eventually converge to $\mathcal{D}_2$.

\section{Iteration of $M_n$ on a Real Orthonormal Basis for $\mathcal{D}_2$}

We now show the underlying structure of $\mathcal{D}_2$ by the introduction of a new set of coordinates. Let 

 \begin{equation}
\label{taumat}
 \tau = \begin{bmatrix}
 &0  \\
 &2\pi/n\\
 &4\pi/n\\
 & .\\
 & .\\
 & .\\
 &2(n-1)\pi/n
\end{bmatrix}
\end{equation}

In matrix notation we can define $\cos(\tau)$ as the matrix with $\cos(\tau_j)$ in the jth row. We define $\sin(\tau)$ analogously. It can be shown \cite{orig} that this is indeed an orthonormal set upon normalizing and defining $$\vec C(\tau) = \sqrt{\frac{2}{n}} \cos(\tau)$$ and $$\vec S(\tau) = \sqrt{\frac{2}{n}} \sin(\tau)$$

We now seek to examine this set $\{ \vec C, \vec S \}$ under iteration of our matrix $M_n$, where it is clear that this set is indeed an orthonormal basis for $\textrm{span}\big(\textrm{Re}(z_2),\textrm{Im}(z_2)\big)$. 

Define $\mu_i = \tau_i + \Delta + \pi / n$ where $\Delta \in \mathbb{R}$ is arbitrary. Noting that $\tau_{i+1} = \tau_i + 2\pi/n = S_n \tau_i$, we have:

$$M_n \vec C(\tau + \Delta) = \sqrt{\frac{2}{n}} \Big( (1-\xi)\cos(\mu_i - \pi/n) + \xi \cos(\mu_i + \pi/n) \Big)$$
$$=\sqrt{\frac{2}{n}} \Big( \cos(\mu_i)\cos(\pi/n) + \sin(\mu_i)\sin(\pi/n) -2\xi\sin(\mu_i)\sin(\pi/n) \Big)$$
$$=\sqrt{\frac{2}{n}} \Big( (1-2\xi)\sin(\mu_i)\sin(\pi/n) + \cos(\mu_i)\cos(\pi/n) \Big)$$

Similarly, we can find that $$M_n \vec S(\tau +\Delta) = \sqrt{\frac{2}{n}} \Big( (2\xi -1)\cos(\mu_i)\sin(\pi/n) + \sin(\mu_i)\cos(\pi/n) \Big)$$

In matrix notation, we have:

$$M_n \vec S(\tau +\Delta) = (2\xi-1)\sin(\pi/n) \vec C(\tau + \Delta + \pi/n) + \cos(\pi/n) \vec S(\tau + \Delta + \pi/n)$$
$$M_n \vec C(\tau +\Delta) = \cos(\pi/n) \vec C(\tau + \Delta + \pi/n) + (1-2\xi) \sin(\pi/n) \vec S(\tau + \Delta + \pi/n)$$

In order to study the iteration of $M_n$ on these two basis elements, it is crucial to first introduce suitable notation.

Let $$\vec C_k = \vec C(\tau + k\pi/n)$$ and $$\vec S_k = \vec S(\tau + k\pi/n)$$

Now let $$\alpha = (2\xi -1)\sin(\pi/n)$$ and $$\beta = \cos(\pi/n)$$ 
Setting $\Delta = 0$, we can now put the above equation in a much more succinct form. We have:

\begin{equation}
\label{recur}
M_n \vec S_0 = \alpha \vec C_1 + \beta \vec S_1
\end{equation}

$$M_n \vec C_0 = \beta \vec C_1 - \alpha \vec S_1$$

In this form our two expressions resemble a recurrence relation, which motivates us to look for a closed form of $M_n^k \vec{S}_0$ and $M_n^k \vec{C}_0$. We see that it in fact satisfies a strikingly symmetric relationship.

\vspace{5mm}

\begin{theorem}
Let $\alpha$ and $\beta$ be defined as above. Let $z = \beta + i\alpha$, where $i$ denotes the imaginary unit. We then have:

\begin{equation}
\label{closed}
M_n^k \vec S_0 = |z|^k \bigg( \cos(k \textrm{arg}(z)) \vec S_k + \sin(k \textrm{arg}(z)) \vec C_k \bigg)
\end{equation}
$$M_n^k \vec C_0 = |z|^k \bigg( \cos(k \textrm{arg}(z)) \vec C_k - \sin(k \textrm{arg}(z)) \vec S_k \bigg)$$
\end{theorem}

\begin{proof}
By induction. Base case is trivial. We proceed inductively:

$$M_n^k \vec S_0 = M_n M_n^{k-1}\vec S_0 = M_n |z|^{k-1} \bigg( \cos((k-1) \textrm{arg}(z)) \vec S_{k-1} + \sin((k-1) \textrm{arg}(z)) \vec C_{k-1} \bigg)$$
$$= |z|^{k-1} \Bigg( \cos \big( (k-1) \textrm{arg}(z)\big) \bigg( \alpha \vec C_k + \beta \vec S_k \bigg) + \sin\big( (k-1) \textrm{arg}(z)\big) \bigg( \beta \vec C_k - \alpha \vec S_k \bigg) \Bigg)$$
Collecting terms, we see the coefficient of $\vec C_k$ is:
$$|z|^{k} \Big( \frac{\alpha}{|z|} \cos \big( (k-1) \textrm{arg}(z)\big) + \frac{\beta}{|z|} \sin\big( (k-1) \textrm{arg}(z)\big) \Big) $$

$$ = |z|^k \Big( \sin \big( \textrm{arg} (z) \big) \cos \big( (k-1) \textrm{arg}(z)\big) + \cos\big( \textrm{arg} (z) \big) \sin\big( (k-1) \textrm{arg}(z)\big) \Big) $$

$$= |z|^k \sin \big( k \textrm{arg}(z) \big)$$

Where the angle sum formula was used for the last step. Likewise we find the coefficient for $\vec S_k$ as:

$$|z|^k \Big( \frac{\beta}{|z|} \cos \big( (k-1) \textrm{arg}(z)\big) - \frac{\alpha}{|z|} \sin\big( (k-1) \textrm{arg}(z)\big) \Big) $$
$$= |z|^k \cos \big( k \textrm{arg}(z) \big)$$
As desired. The case for $M_n^k \vec C_k$ is nearly identical.
\end{proof}

\vspace{5mm}

\begin{remark}
If we rewrite the results given above a little differently, the elegance and symmetry of the two expressions becomes more striking. We have:

$$M_n^k \vec S_0 = \textrm{Re}(z^k) \vec S_k + \textrm{Im}(z^k) \vec C_k$$
$$M_n^k \vec C_0 = \textrm{Im}(\bar z^k) \vec S_k + \textrm{Re}(\bar z^k) \vec C_k$$
\end{remark}

\section{Tracking the Vertex Vectors}

We now summarize a collection of the above results to continue below. Without loss of generality, we can assume our given polygon has centroid at the origin. Under this assumption, we then know that both vectors $\vec x$ and $\vec y$ are orthogonal to $z_1 = e$, the $nx1$ vector of all 1's. Thus:
$$\vec x, \vec y \in \textrm{span}(z_2,z_3, z_4, ...,z_n) = \textrm{span}(\vec C, \vec S, z_4, ... , z_n)$$

From here, we can express $\vec x^{(k)}$ and $\vec y^{(k)}$ as:
\begin{equation}
\label{vert1}
\vec x^{(k)} = \vec u^{(k)} + \mathcal{O}(\rho_n^k)
\end{equation}
$$\vec y^{(k)} = \vec v^{(k)} + \mathcal{O}(\rho_n^k)$$

where 
\begin{equation}
\label{unitvert}
\vec u^{(k)} = \frac{\zeta_1 M_n^k \vec C_0 + \eta_1 M_n^k \vec S_0}{|| \zeta_1 M_n^k \vec C_0 + \eta_1 M_n^k \vec S_0||}
\end{equation}

$$\vec v^{(k)} = \frac{\zeta_2 M_n^k \vec C_0 + \eta_2 M_n^k \vec S_0}{|| \zeta_2 M_n^k \vec C_0 + \eta_2 M_n^k \vec S_0||}$$

and $\zeta_i$, $\eta_i$ are arbitrary scalars.

Let us now assume that $\vec u^{(0)}$ is a given unit vector such that $$\vec u^{(0)} = \cos(\theta_u)\vec C_0 + \sin(\theta_u)\vec S_0$$

We want to examine the behavior of $M_n^k \vec u^{(0)}$. Using the results of the previous section, we can now do this and find a rather simple closed form.

\begin{equation}
\label{closedvert}
M_n^k u^{(0)} = \bigg(  \textrm{Im}(\bar z^k) \vec S_k + \textrm{Re}(\bar z^k) \vec C_k \bigg)\cos(\theta_u) + \bigg(\textrm{Re}(z^k) \vec S_k + \textrm{Im}(z^k) \vec C_k \bigg)\sin(\theta_u)
\end{equation}
$$= |z|^k \bigg( \cos\big( \theta_u - k \textrm{arg}(z))\vec C_k +  \sin\big( \theta_u - k \textrm{arg}(z))\vec S_k \bigg)$$

Similarly, we see 
\begin{equation}
\label{closedvert2}
M_n^k v^{(0)} =|z|^k \bigg( \cos\big( \theta_v - k \textrm{arg}(z))\vec C_k +  \sin\big( \theta_v - k \textrm{arg}(z))\vec S_k \bigg)
\end{equation}
With this result we are now able to prove the following lemma.

\vspace{5mm}

\begin{lemma}
For any arbitrary point of division $\xi \in (0,1)$, iteration of $M_n$ onto the vertices of a given polygon will converge to the centroid of the polygon, and every polygon in the sequence will have the same centroid.
\end{lemma}

\begin{proof}
From Lemma 2.1 we have already shown that if our sequence converges to a point, then the hypothesis holds. We must then show that our vertices will in fact converge to a point. 
Since we have already shown our vertex vectors $M_n^k\vec x$, $M_n^k \vec y$  will become arbitrarily close to $M_n^k \vec u^{(0)}$ and $M_n \vec v^{(0)}$, respectively, we only need to show convergence of the latter two vectors.

From \eqref{closedvert} and \eqref{closedvert2}, we then deduce:
$$||M_n^k u^{(0)}|| = |z|^k = \Big( (2\xi -1)^2 \sin^2(\pi/n) + \cos^2 (\pi/n) \Big) ^{\frac{k}{2}}$$
For $\xi \in (0,1)$, $(2\xi-1)^2 < 1$. Thus, 

$$\Big( (2\xi -1)^2 \sin^2(\pi/n) + \cos^2 (\pi/n) \Big) ^{\frac{1}{2}}$$
$$< \Big( \sin^2(\pi/n) + \cos^2 (\pi/n) \Big) ^{\frac{1}{2}} = 1$$

Therefore $||M_n^k \vec u^{(0)}|| \rightarrow 0$ for all $\xi \in (0,1)$ (similarly for $\vec v$), so our sequence converges to a single point, the centroid.
\end{proof}

\vspace{5mm}

\textbf{Corollary}. The sequence of polygons $\mathcal{P}(M_n^k \vec x^{(0)}, M_n^k \vec y^{(0)})$ such that $\xi = 1/2$ will converge most quickly to the centroid of $\mathcal{P}(\vec x^{(0)}, y^{(0)})$.

\vspace{5mm}

\begin{proof}
We note that $(2\xi -1)^2$ is positive everywhere except its root $\xi = 1/2$. From the work of the above lemma, we have: 

$$||M_n^k u^{(0)}|| \geq \cos^k(\pi/n)$$

With equality if and only if $\xi = 1/2$.
\end{proof}

\section{The Limiting Ellipse}

We now turn our attention back to the behavior of our normed vectors $\vec u^{(k)}$ and $\vec v^{(k)}$. From \eqref{closedvert} and \eqref{closedvert2}, we see that the magnitudes of these vectors are only dependent on the magnitude $|z|^k$, and so we easily deduce the normed vertex vectors:

\begin{equation}
\label{unitvertex}
\vec u^{(k)} = \cos\big( \theta_u - k \textrm{arg}(z))\vec C_k +  \sin\big( \theta_u - k \textrm{arg}(z))\vec S_k
\end{equation}

$$\vec v^{(k)} = \cos\big( \theta_v - k \textrm{arg}(z))\vec C_k +  \sin\big( \theta_v - k \textrm{arg}(z))\vec S_k$$

From \cite{orig} we find that our ellipse can be completely specified by the singular value decomposition (SVD) of the set given by:

$$\begin{bmatrix}
 & \vec u^{(k)} \\
 & \vec v^{(k)}
\end{bmatrix} =
\sqrt{\frac{2}{n}}
\begin{bmatrix}
&\cos\big( \theta_u - k \textrm{arg}(z)) & \sin\big( \theta_u - k \textrm{arg}(z)) \\
&\cos\big( \theta_v - k \textrm{arg}(z)) & \sin\big( \theta_v - k \textrm{arg}(z))
\end{bmatrix}
\begin{bmatrix}
&\cos(t_i) \\
& \sin(t_i)
\end{bmatrix}$$

where $t_i = \tau_i + k\pi/n$. We now seek to find the SVD of our 2x2 matrix. This is because, if we put a matrix $A = U \Sigma V^T$, then the matrices $U$ and $V$ will specify the tilt of our ellipse, whose semiaxes are specified by the diagonal elements of our matrix $\Sigma$. 
\clearpage

\begin{theorem} If 
$$A = \begin{bmatrix}
&\cos\big( \theta_u - k \textrm{arg}(z)) & \sin\big( \theta_u - k \textrm{arg}(z)) \\
&\cos\big( \theta_v - k \textrm{arg}(z)) & \sin\big( \theta_v - k \textrm{arg}(z))
\end{bmatrix} = 
\begin{bmatrix}
& a_{11} & a_{12} \\
& a_{21} & a_{22}
\end{bmatrix}
$$

Then there exist matrices $U$, $\Sigma$, and $V$ such that $A = U \Sigma V^T$, where

\begin{equation}
\label{U}
U = 
\begin{bmatrix}
&\cos(\pi/4) & -\sin(\pi/4) \\
&\sin(\pi/4) & \cos(\pi/4)
\end{bmatrix}
\end{equation}

\begin{equation}
\label{sigma}
\Sigma = \sqrt{\frac{2}{n}}
\begin{bmatrix}
&\Big( 1+\sin \big( \theta_u+\theta_v - 2k\textrm{arg}(z)\big) \Big)^{\frac{1}{2}}& 0 \\
& 0 & \Big( 1-\sin \big( \theta_u+\theta_v - 2k\textrm{arg}(z)\big) \Big)^{\frac{1}{2}}
\end{bmatrix}
\end{equation}

\begin{equation}
\label{V}
V = \begin{bmatrix}
& \textrm{sgn}(a_{11} + a_{12} + a_{21} + a_{22})\cos(\pi/4) & -\textrm{sgn}(a_{11} - a_{12} - a_{21} + a_{22})\sin(\pi/4) \\
& \textrm{sgn}(a_{11} + a_{12} + a_{21} + a_{22})\sin(\pi/4) & \textrm{sgn}(a_{11} - a_{12} - a_{21} + a_{22})\cos(\pi/4)
\end{bmatrix}
\end{equation}
\end{theorem}

\vspace{5mm}
\begin{proof}
It can be shown \cite{linalg} that the singular value decomposition for any 2x2 matrix is of the form $$A =
\begin{bmatrix}
& \cos(\theta) & -\sin(\theta) \\
& \sin(\theta) & \cos(\theta)
\end{bmatrix}
\begin{bmatrix}
& \sigma_1 & 0 \\
& 0 & \sigma_2
\end{bmatrix}
\begin{bmatrix}
&\cos(\phi) & -\sin(\phi) \\
&\sin(\phi) & \cos(\phi)
\end{bmatrix}$$

The rest is omitted as it is not essential to understanding the properties of our ellipse. The explicit forms are found by straightforward computation of the explicit forms of the SVD of a 2x2 matrix (see \cite{linalg}). 
\end{proof}

Firstly, we see that by the matrices \eqref{U} and \eqref{V}, our ellipse will still be oriented at 45 degrees no matter what $\xi \in (0,1)$ is chosen.

One of the properties to be analyzed can be deduced from the matrix \eqref{sigma}: periodic semi-axes.
This can be seen by the explicit forms of each semi-axes of our ellipse, which are given by:
$$\textrm{Axis}_1 = \sqrt{\frac{2}{n}} \sigma_1$$
$$\textrm{Axis}_2 = \sqrt{\frac{2}{n}} \sigma_2$$

\begin{figure}[H]
    \centering
\includegraphics[scale=0.48]{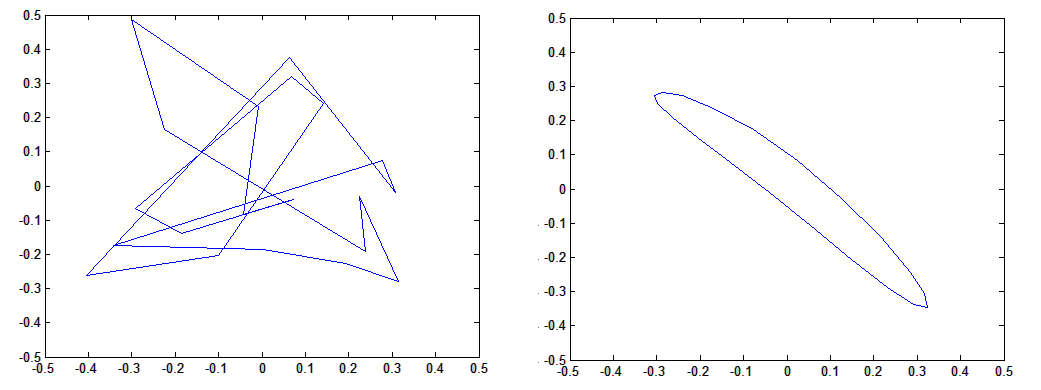}

\includegraphics[scale=0.48]{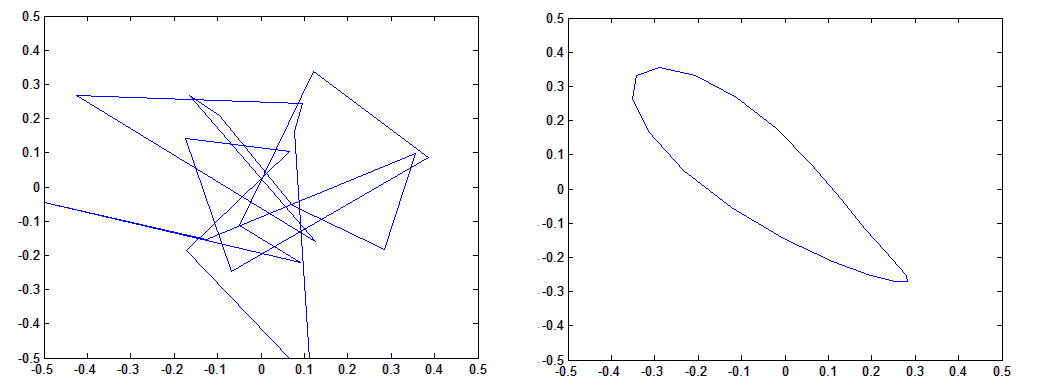}

\includegraphics[scale=0.48]{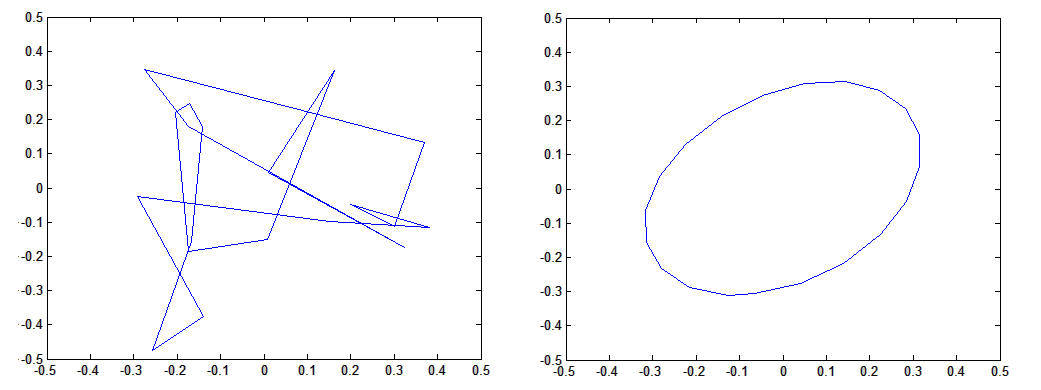}
\vspace{1mm}
\caption{Here we see 100 iterations of our normed transformation with 20 points. From top to bottom $\xi=1/5$, $1/4$, and $2/5$, respectively. We see that each $\xi$ tends to the ellipse at $45$ degrees, and values closer to $1/2$ converge more quickly.}

\end{figure}

\section{Periodicity of the $\mathcal{D}_2$ Limiting Ellipse}

Recall that 
$$\vec u^{(k)} = \cos\big( \theta_u - k \textrm{arg}(z))\vec C_k +  \sin\big( \theta_u - k \textrm{arg}(z))\vec S_k$$
$$\vec v^{(k)} = \cos\big( \theta_v - k \textrm{arg}(z))\vec C_k +  \sin\big( \theta_v - k \textrm{arg}(z))\vec S_k$$

By the periodicity of the trigonometric functions, we deduce an interesting property of the $\mathcal{D}_2$ ellipse.

\vspace{5mm}

\begin{lemma}
For $z = \beta + i\alpha$, where $\alpha$ and $\beta$ defined as before, $\vec u^{(k)}$ is periodic with period $2q$ if and only if $\textrm{arg}(z)$ is a rational multiple of $\pi$.
\end{lemma}

\vspace{5mm}
\begin{proof}
Assume $\vec u^{(k)} = \vec u^{(k+2q)}$, where $q \in \mathbb{Z}$. We have: 
$$\vec u^{(k)} = \cos\big( \theta_u - (k+2q) \textrm{arg}(z))\vec C_k +  \sin\big( \theta_u - (k+2q) \textrm{arg}(z))\vec S_k$$
From here we see that $\sin(2q \textrm{arg}(z))$ must vanish after the use of elementary trigonometric identities. Thus $$\sin(2q \textrm{arg}(z)) = 0 \Rightarrow \textrm{arg}(z) = \frac{2p}{2q}\pi = \frac{p}{q}\pi$$
where $p$, $q \in \mathbb{Z}$. Also, $\vec S_k \rightarrow \vec S_{k+2q} = S_n^q\vec S_k$, and likewise for $\vec C_k$. Therefore we see that this is periodic since the upshift matrix is just a reordering of our entries. The converse is trivial.
\end{proof}

Although this is a rather idealized case, in practice it seems that there still exists some periodicity, or \textit{almost periodicity}, for any point of division. This stems from the fact that for any given argument of $z$ as defined, it can be approximated arbitrarily close to some rational multiple of $\pi$. Explicitly, we have:
$$ k \textrm{arg}(z) = p \pi$$ where we allow that $p$ be any arbitrary integer. We do not hold $p$ to be fixed, we allow it to vary. Because of this, as $k \to \infty$, $\textrm{arg}(z)$ will be sufficiently close to a rational multiple of $\pi$ an arbitrary amount of times.

To illustrate this more clearly, consider taking $20$ randomized vertices with $\xi = 1/4$. Then, we first start taking $\pi / \textrm{arg}(z) \approx 13.4686$. We now scale this by some integer such that our value is very close to an integer. In this case, it is clear that scaling by $2$ gets us close to the number $27$, so we could expect that this is weakly periodic with a period of $27$. Continued scaling shows us other places where we get even closer to integer values. For instance, scaling by $32$ gives us a number extremely close to $431$, and it is to be expected that this process can be carried out indefinitely for closer and closer approximations.

In fact, this is an illustration of \textit{Dirichlet's Approximation Theorem}, which gives us the following fact:

Given the number $\pi$, there exists an infinite sequence of rational numbers such that $$\Big| \pi - \frac{p}{q} \Big| < \frac{1}{q^2}$$
Of course, this is just a particular case, and Dirichlet's Theorem has a much more general statement.

Finally, we consider our semi-axes. As is clearly seen by our matrix \eqref{sigma}, we have semi-axes that are dependent on $k$. This leads us to deduce a property of oscillating semi-axes. In practice, however, this oscillation is very slight and almost impossible to see for larger $n$, as can be predicted by the fact that our periodicity gets much larger as the number of vertices grows.

\vspace{5mm}

 We now consider the case of taking a different point of division at each segment.

\section{Taking Different Points of Division for Each Segment of $\mathcal{P}(\vec x^{(0)},\vec y^{(0)})$}

We can now use the results given so far to generalize even further. Suppose that we are given a polygon $\mathcal{P}(\vec x^{(0)},\vec y^{(0)})$ and instead of using a fixed point of division $\xi$ for every segment, we give an arbitrary and different point of division $\xi_i$ to the $i$th segment of our polygon. This case is inherently similar, and can in fact be reduced to the properties of the sequences of polygons $\{ \mathcal{P}_{\xi_i}(\vec x^{(k)}, \vec y^{(k)}) \}$, where $\mathcal{P}_{\xi_i} (\vec x^{(k)}, \vec y^{(k)})$ denotes the sequence of polygons produced by our iteration when using the fixed point of division $\xi_i$ for each segment.
In this way, it is easy to see that we can express our random point of division polygon, denoted $\tilde{\mathcal{P}} (\vec x^{(k)}, \vec y^{(k)})$ as:
$$\tilde{\mathcal{P}} (\vec x^{(k)}, \vec y^{(k)}) = \cap_{i = 1}^{n} \mathcal{P}_{\xi_i}(\vec x^{(k)}, \vec y^{(k)})$$

 With this form, we see that iteration of $M_n$ on the vertex vectors of $\tilde{\mathcal{P}} (\vec x^{(k)}, \vec y^{(k)})$ without normalization will will also converge to the centroid of the original polygon, since we have already proved that every polygon in the intersection has to converge to the centroid.

We now consider the case of $\tilde{\mathcal{P}} (\vec u^{(k)}, \vec v^{(k)})$, which is the normalized iteration. It can be easily shown that our periodicity will exist if and only if $\textrm{arg}(z)$ is a rational multiple of $\pi$ and that our period will become $\textrm{LCM}([q_1, q_2, ..., q_n]$, where $q_i$ is the period of the ellipse produced by $\xi_i$.

\section*{Acknowledgements}
I would like the thank Jimmy Dillies for introducing me to this problem, as well as Scott Kersey for aiding me in creating the MATLAB code to model this problem.


\begin{thebibliography}{10}

\bibitem{orig} Elmachtoub, A.N. and van Loan, C.F., \textit{From Random Polygon to Ellipse: An Eigenanalysis}, SIAM Review, 2010, vol. 52, no. 1, pp. 151-170

\bibitem{linalg}
Randy Ellis, Singular Value Decomposition.

Retrieved from \texttt{http://www.lucidarme.me/?p=4624} on April 1, 2016

\bibitem{3} Jiu Ding, L Richard Hitt, Xin-Min Zhang, \textit{Markov chains and dynamic geometry of polygons}, Linear Algebra and its Applications, Volume 367, 1 July 2003, Pages 255-270, ISSN 0024-3795, http://dx.doi.org/10.1016/S0024-3795(02)00634-1.

\bibitem{4} Volkov, S. (2013), \textit{Random geometric subdivisions}. Random Struct. Alg., 43: 115–130. doi: 10.1002/rsa.20454

\bibitem{5} Dan Ismailescu and Jacqueline Jacobs, \textit{On sequences of nested triangles}, Periodica Mathematica
Hungarica 53 (2006), no. 1-2, 169–184.



\end{thebibliography}
\end{document}